\documentclass[12pt,reqno]{amsart}
\usepackage[utf8]{inputenc}

\usepackage{amsfonts, amssymb, amsmath, amsthm, mathtools, thmtools}
\usepackage{mathrsfs} 
\usepackage{latexsym}
\usepackage{enumerate}
\usepackage{multicol}
\usepackage{times}
\usepackage{verbatim}
\usepackage{tikz-cd}
\usepackage{fullpage}
\usepackage{here}
\usepackage{url}
\usepackage{csquotes}
\usepackage{placeins}
\usepackage{epic}
\usepackage{graphicx}
\usepackage{epstopdf}
\usepackage{pstricks}
\usepackage{pst-plot}
\usepackage{tikz}
\usepackage{xcolor}
\usepackage{graphicx}
\usepackage{subcaption}
\newenvironment{poc}{\begin{proof}[Proof of claim]}{\end{proof}}

\usetikzlibrary{shapes.geometric, positioning, calc}

\input xypic
\xyoption{all}
\xyoption{poly}

\usepackage[colorlinks=true]{hyperref}
\hypersetup{citecolor=blue, linkcolor=blue}
\usepackage[capitalise]{cleveref}

\crefname{equation}{}{}
\crefname{figure}{{\sc Figure}}{{\sc Figure}}
\crefname{subsection}{Subsection}{Subsections}

\usetikzlibrary{matrix,arrows,decorations.pathmorphing}

\usepackage[euler]{textgreek}
\usepackage{tikz,tikz-cd}
\usepackage[colorinlistoftodos, textwidth = 2.3cm]{todonotes}
\newcommand{\nc}{\newcommand}

\newtheorem{theorem}{Theorem}[section]
\newtheorem{proposition}[theorem]{Proposition}
\newtheorem{lemma}[theorem]{Lemma}
\newtheorem{corollary}[theorem]{Corollary}
\newtheorem{conjecture}[theorem]{Conjecture}
\newtheorem{claim}[theorem]{Claim}

\newtheorem*{claim*}{Claim}

\theoremstyle{definition}

\newtheorem{definition}[theorem]{Definition}
\newtheorem{remark}[theorem]{Remark}

\newcommand{\F}{{\mathbb F}}

\newcommand{\Z}{{\mathbb Z}}

\usepackage{tabstackengine}
\stackMath

\numberwithin{equation}{section} 
\numberwithin{figure}{section}
\numberwithin{table}{section}

\nc{\Beaver}[1]{\todo[size=\tiny,color=cyan!10]{#1 \\ \hfill --- Beaver}}
\nc{\BEAVER}[1]{\todo[size=\tiny,inline,color=cyan!10]{#1
		\\ \hfill --- Beaver}}
  
\nc{\Semin}[1]{\todo[size=\tiny,color=magenta!10]{#1 \\ \hfill --- Semin}}
\nc{\SEMIN}[1]{\todo[size=\tiny,inline,color=magenta!10]{#1
		\\ \hfill --- Semin}}

\nc{\nt}[1]{\todo[size=\tiny,color=exgreen!10]{#1 \\ \hfill --- Note}}
\nc{\NT}[1]{\todo[size=\tiny,inline,color=exgreen!10]{#1
		\\ \hfill --- Note}}

\title{Product representations of polynomials over finite fields}

\author{Hyunwoo Lee}
\address{Department of Mathematical Sciences, KAIST \\ and Extremal Combinatorics and Probability Group, Institute for Basic Science \\ Daejeon\\ South Korea}
\email{hyunwoo.lee@kaist.ac.kr}
\author{Chi Hoi Yip}
\address{Department of Mathematics, Hong Kong University of Science and Technology, Clear Water Bay, Hong Kong}
\email{machyip@ust.hk}
\author{Semin Yoo}
\address{Discrete Mathematics Group \\ Institute for Basic Science \\ 55 Expo-ro Yuseong-gu, Daejeon 34126 \\ South Korea}
\email{syoo19@ibs.re.kr}

\begin{document}
 
\begin{abstract}
Erd\H{o}s, S\'ark\"ozy, and S\'os studied the asymptotics of the maximum size of a subset of $\{1,2,\ldots, N\}$ such that it does not contain $k$ distinct elements whose product is a perfect square. More generally, Verstra\"ete proposed a conjecture regarding the asymptotic behavior of the same quantity with the set of perfect squares replaced by the value set of a polynomial in $\mathbb{Z}[x]$. In this paper, we study a finite field analogue of Verstra\"ete's conjecture.
\end{abstract}

\maketitle

\section{Introduction}
Throughout the paper, let $q$ be a prime power. Let $\F_q$ be the finite field with $q$ elements, and $\F_q^*$ be its multiplicative group. For a positive integer $n$, we use $\Z_n$ to denote the cyclic group $\Z/n\Z$.

In 1995, Erd\H os, S\'ark\"ozy, and S\'os \cite{ESS95} initiated the study of product representations of squares. More precisely, for each integer $k\geq 2$ and positive integer $N$, let $F_k(N)$ be the maximum size of a subset $A$ of $\{1,2,\ldots, N\}$ such that it does not contain $k$ distinct elements in $A$ whose product is a perfect square. They showed that the asymptotic behavior of $F_k(N)$ strongly depends on the parity of $k$. Moreover, they determined the asymptotics of $F_k(N)$ for $k$ even and $k=3$. There have been several refinements to the bounds on $F_k(N)$ for even $k\geq 4$; see, for example, the work of Gy\H{o}ri \cite{Gyori95}, Naor--Verstra\"{e}te \cite{NV08}, Pach \cite{P15, P19}, and Pach--Vizer \cite{PV23}. 

Granville and Soundararajan \cite{GS01} studied a variant of this problem: they provided an asymptotic formula for the maximum size of a subset $A$ of $\{1,2,\ldots, N\}$ such that $A$ does not contain an odd number of distinct elements whose product is a perfect square. Recently, Tao \cite{T25} made important progress on $F_k(N)$ for $k\geq 5$ odd. 

More generally, in 2006,  Verstra\"ete \cite{V06} studied product representations of polynomials. For a polynomial $h \in \Z[X]$, and positive integers $k,N$, let $F_k(N;h)$ be the maximum size of a subset $A$ of $\{1,2,\ldots, N\}$ such that it does not contain $k$ distinct elements in $A$ whose product is in the value set of $h$, that is, there are no distinct $a_1,a_2,\ldots, a_k\in A$, such that $a_1a_2\cdots a_k=h(x)$ for some $x\in \Z$. He formulated the following conjecture:

\begin{conjecture}[{\cite[Conjecture 3]{V06}}]\label{conj:V}
Let $h \in \Z[X]$ and let $k$ be a positive integer. 
Then, for some positive constant $\rho=\rho(k,h)$ depending only on $k$ and $h$, either $F_k(N;h)/N \to \rho$ or $F_k(N;h)/\pi(N) \to \rho$ as $N \to \infty$.
\end{conjecture}

Verstra\"ete \cite{V06} made some partial progress towards \cref{conj:V}. In particular, he confirmed the conjecture for a family of polynomials. Recently, Fleiner, Juh\'asz, K\"ov\'er, Pach, and S\'andor \cite{FJKPS} studied product representations of perfect cubes, that is, they studied the asymptotics of $F_k(N;x^3)$. 
In particular, they disproved \cref{conj:V} by showing that $$F_6(N;x^3)=\big(1+o(1)\big) \frac{N\log \log N}{\log N}.$$
Very recently, Pach and S\'andor \cite{PS26} studied a related variant in which one requires that the product of every nonempty subset of the given set is not a perfect $d$-th power. More precisely, let $\rho_d(N)$ denote the maximum size of a subset $A \subseteq \{1,2,\ldots,N\}$ such that no product of a nonempty collection of distinct elements of $A$ is a perfect $d$-th power. 
They proved that, for every integer $d\ge2$, 
\[
\rho_d(N)=\sum_{j=1}^{d-1}\pi(N/j) + O_d\big(\pi(\sqrt{N})\big),
\]
and moreover, when $d$ is a prime power, they obtained the exact formula
\[
\rho_d(N)=\sum_{j=1}^{d-1}\pi(N/j)
\]
for all sufficiently large $N$.

\medskip

Motivated by the results discussed above, in this paper, we study finite field analogues of \cref{conj:V}. Before stating our main results, we first introduce the finite field analogue of $F_k(N;h)$ and provide some additional background. 

\begin{definition}\label{defn}
For a positive integer $k$, a prime power $q$, and a non-constant polynomial $h\in \F_q[x]$, define
\[
F_k(q;h):=\max\{|A|: A\subset \F_q^* \text{ and $A$ satisfies }(\star)\},
\]
where $(\star)$ denotes that for all distinct $a_1,\dots,a_k\in A$ and all $x\in\F_q$, 
$a_1\cdots a_k \neq h(x)$.    
\end{definition}

We remark that for some special polynomials $h$, the quantity $F_k(q;h)$ has been well-studied. 
For example, if $q$ is an odd prime power, $\alpha$ is a non-square in $\F_q$, and $h(x)=\alpha x^2-1$, then $F_k(q;h)$ is ``essentially" the same as the maximum size of a $k$-Diophantine tuple over $\F_q$ \cite{YY25}. See for example \cite{Gyarmati01, Shparlinski23, YY25} for further discussions. Recall that if $q$ is an odd prime power, we say $A \subset \F_q^*$ is a \emph{$k$-Diophantine tuple over $\F_q$} if $a_1a_2\cdots a_k+1$ is a square in $\F_q$ for any distinct $a_1,\dots,a_k\in A$. If $A \subset \F_q^*$ is a \emph{$k$-Diophantine tuple over $\F_q$} with the additional property that $a_1\cdots a_k=-1$ for any distinct $a_1,\dots,a_k\in A$, then $A$ satisfies $(\star)$ in \cref{defn}. Conversely, if $A$ satisfies $(\star)$ in \cref{defn}, then for any distinct $a_1,\dots,a_k\in A$, we have 
$a_1\cdots a_k\notin h(\F_q)$, and thus $a_1a_2\cdots a_k+1$ is a square in $\F_q^*$; thus $A$ is necessarily a $k$-Diophantine tuple over $\F_q$.

In the spirit of \cref{conj:V}, our goal is to determine an asymptotic formula for $F_k(q;h)$, under the natural assumption that $q$ is large compared to $k$ and the degree of $h$. 
To state our main theorem, we define one more terminology. Let $k,n$ be positive integers with $k\geq 2$. 
For a subset $B\subset\Z_n$, its \emph{$k$-fold sumset} is 
$$kB:=\{b_1+\cdots+b_k: b_i\in B \text{ for } 1\leq i \leq k\}.$$ For each $s\in \Z_n$, we define
\[
m(k,n;s):=\max\big\{|B|: B\subset\Z_n,\ s\notin kB\big\},
\]
that is, $m(k,n;s)$ is the maximum size of a subset $B$ of $\Z_n$ such that the $k$-fold sumset $kB$ avoids $s$.
Observe that if $\lambda$ is a unit in $\Z_n$, that is, $\gcd(\lambda,n)=1$, then $m(k,n;s)=m(k,n;\lambda s)$. 
We now state our main theorem.

\begin{theorem}\label{thm:main}
Let $k\geq 2$ be an integer, $q$ be a prime power, and $h\in \F_q[x]$ be a non-constant polynomial. 
Let $\ell$ be the largest positive integer such that $h=Cf^{\ell}$ for some $C\in \F_q^*$ and $f\in \F_q[x]$. Let $n=\gcd(\ell, q-1)$ and let $H$ be the subgroup of $\F_q^*$ of index $n$. Let $g$ be a generator of $\F_q^*$ and let $s$ be the unique integer in $\{0,1,\ldots, n-1\}$ such that $C\in g^sH$. Then
    \[F_k(q;h)=\frac{m(k,n;s)}{n}\cdot q+O(\sqrt{q}),\]
where the implied constant of the error term depends only on $k$ and the degree of $h$.
\end{theorem}

The proof of our main theorem is given in \cref{sec:proof}. In fact, our proof shows a stronger conclusion: if $m(k,n;s)>0$, then extremal sets are essentially unions of $m(k,n;s)$ cosets of $H$. We refer to \cref{rem:structure} for a precise statement. 

We make a few remarks on the theorem below. 

\begin{remark}\label{rmk:rmk}\
\begin{enumerate}
    \item The main term in \cref{thm:main} is independent of the chosen factorization of $h$ and of the chosen generator $g$. 
    
    Indeed, suppose we have two factorizations $h=Cf^{\ell}=\widetilde{C}\widetilde{f}^{\ell}$, and two generators $g$ and $\widetilde{g}$. Let $s,s'\in \{0,1,\ldots, n-1\}$ such that $C\in g^s H$ and $\widetilde{C} \in \widetilde{g}^{\widetilde{s}}H$.     Then $C$ and $\widetilde{C}$ are in the same coset of $H$, and thus $g^s H=\widetilde{g}^{\widetilde{s}}H$.
Assume that $\widetilde{g}=g^u$, where $\gcd(u,q-1)=1$. It follows $g^s/\widetilde{g}^{\widetilde{s}}=g^{s-u\widetilde{s}}\in H$, that is, $s\equiv u\widetilde{s} \pmod n$. Since $\gcd(u,n)=1$, we have $m(k,n;s)=m(k,n;s')$ by the observation before the theorem.

 \item When $m(k,n;s)>0$, \cref{thm:main} yields
    \[ F_k(q;h)=\frac{m(k,n;s)}{n} \cdot q+O(\sqrt q),\]
    so in this case $F_k(q;h)$ has a linear main term in $q$. In particular, this confirms a finite-field analogue of Verstra\"ete's conjecture (\cref{conj:V}).
    
    The condition $m(k,n;s)>0$ can be characterized by the following equivalence: 
    \[n \mid k \text{ and }s=0 \qquad \Leftrightarrow \qquad m(k,n;s)=0.\]
    Suppose that $n \mid k$ and $s=0$, $B$ is nonempty, let $b\in B$. Then $kb \in kB$. Since $n \mid k$, $kb \equiv 0\pmod n$, so $0 \in kB$, a contradiction. Thus $B=\varnothing$, so $m(k,n;0)=0$. Conversely, we prove the contrapositive. Assume that $n\nmid k$ or $s\neq 0$. We claim that $m(k,n;s)>0$. If $s\neq 0$, take $B=\{0\}$. Then $kB=\{0\}$, so $s\notin kB$, and hence $m(k,n;s)\ge 1$. Now assume that $s=0$ and $n\nmid k$. Take $B=\{1\}$. Then $kB=\{k\}$, and since $n\nmid k$ we have $k\not\equiv 0\pmod n$, so $0\notin kB$. Thus $m(k,n;0)\ge 1$. Therefore, in either case, we have $m(k,n;s)>0$. 
    
    \item When $m(k,n;s)=0$, \cref{thm:main} implies that $F_k(q;h)=O(\sqrt q)$. 
    In fact, the $O(\sqrt q)$ bound can sometimes be sharp by considering the following example. 
    
    Let $q=p^2$ with $p$ odd, and consider the quadratic polynomial $h(x)=\alpha x^2 -1 \in \F_{q}[x]$ for some nonsquare $\alpha\in \F_{q}^*$. Since $\ell=1$, $n=\gcd(\ell,q-1)=1$. Thus, $m(k,1;0)=0$.
    
    Let $u$ be a generator of $\F_p^*$, so that $\F_p^*=\langle u\rangle$ has order $p-1$ and $-1=u^{(p-1)/2}$. 
    Set
    \[
    t:=\left\lfloor \frac{p-1}{2k}\right\rfloor+1 \qquad \text{and} \qquad
    A:=\{1,u,\dots,u^{t-1}\}\subset \F_p^*\subset \F_{p^2}^*.
    \]
    Then $|A|=t=\Theta(p)=\Theta(\sqrt q)$.
    Moreover, for any $k$ distinct elements $a_1,\dots,a_k\in A$, denoting $a_j=u^{b_j}$ with $b_j\in\{0,1,\dots,t-1\}$, we have $a_1\cdots a_k=u^{b_1+\cdots+b_k}$. 
    Since the exponents $b_1,\ldots,b_k$ are distinct and lie in $\{0,1,\ldots, t-1\}$, for $k \ge 2$, 
    \[b_1+\cdots+b_k \le (t-1) + \cdots + (t-k)=k(t-1)-\frac{k(k-1)}{2}<\frac{p-1}{2}. \]
    and so $b_1+\cdots+b_k\not \equiv (p-1)/2 \pmod{p-1}$.
    Thus, we have $a_1\cdots a_k\neq -1$.
     
    Note that $\F_p^*\subset (\F_{p^2}^*)^2$.
    Indeed, $\F_p^*=\left<g^{p+1} \right>$ for some generator $g$ of $\F_{p^2}^*$ since the order of $g^{p+1}$ is $p-1$. Since $p+1$ is even, the generator $g^{p+1}$ of $\F_p^*$ is a square in $\F_{p^2}^*$. 
    Moreover, since $\F_p$ is a subfield of $\F_{p^2}$, for any distinct $a_1,\ldots,a_k \in A$, we have $a_1\cdots a_k+1 \subset \F_p\subset (\F_{p^2}^*)^2 \cup \{0\}$. Thus, for any $x \ne 0$, we have $a_1\cdots a_k +1 \ne \alpha x^2$.
    Together with $a_1\cdots a_k\ne -1$, which rules out the case $x=0$, we conclude that $a_1\cdots a_k \ne h(x)$ for all $x\in \F_{p^2}$. It follows that $F_k(q;h)=\Theta( \sqrt q)$.
   
   In general, it seems challenging to determine the asymptotic behavior of $F_k(q;h)$, and we believe that establishing the full analogue of \cref{conj:V} over finite fields is completely out of reach. As an example, consider the same $h$ as above, where $q$ is an odd prime power but not necessarily a square. As discussed in the introduction, in this case, $F_k(q;h)$ is upper bounded by the maximum size of a $k$-Diophantine tuple over $\F_q$. However, the best-known general lower bound on the latter quantity is $\Omega(\log q)$, while the best-known upper bound is $O(\sqrt{q})$; see \cite{YY25} and the references therein.    
    \item  The remarks above show that $F_k(q,h)$ could be of the order $\Theta(q)$ or $\Theta(\sqrt{q})$, but cannot be of the order $\Theta(q^{\theta})$ for $1/2<\theta<1$. A natural question arises: is it possible to construct $h$ and $q$ such that $F_k(q,h)$ is $\Theta(q^{\theta})$ for some $0<\theta<1/2$? We conjecture that the answer is positive if $\theta=1/m$ for each positive integer $m\geq 3$, but we are not able to prove it. Nevertheless, below, we suggest a plausible construction. 

    Let $m\geq 3$ be a positive integer. Let $q=p^m$ with $p$ a prime such that $p \equiv 1 \pmod m$. Note that $\F_p$ is the largest subfield of $\F_q$ contained in $(\F_q)^m$. To see that, if $\F_{p^r}$ is a subfield of $\F_q$ such that it is contained in $(\F_q)^m$, then we must have $r \mid m$ and $m \mid \frac{q-1}{p^r-1}$; however, since $p\equiv 1 \pmod m$, we have
    $$
    \frac{q-1}{p^r-1}=\sum_{j=0}^{m/r-1} p^{jr} \equiv \frac{m}{r} \pmod m,
    $$
    that is, $r=1$. 

    Consider the polynomial $h(x)=\alpha x^m -1 \in \F_{q}[x]$ for some $\alpha\in \F_{q} \setminus (\F_q)^m$.   
    Since $\ell=1$, $n=\gcd(\ell,q-1)=1$. Thus, $m(k,1;0)=0$.
    By a similar argument as above in (3), we can pick $A \subset \F_p$ with $|A|\geq p/2k$ so that for any distinct $a_1,a_2,\ldots, a_k\in A$, we have $a_1a_2\cdots a_k \neq h(x)$ for any $x\in \F_q$. This shows that    
    $F_k(q;h)=\Omega(q^{1/m})$. We suspect that in this setting, we actually have $F_k(q;h)=\Theta(q^{1/m})$.
\end{enumerate}
\end{remark}

Finally, in \cref{sec: sec4}, we estimate $m(k,n;s)$ for each $s\in \Z_n$, where $k\ge2$ and $n$ are positive integers. In particular, when $\gcd(k,n)=1$, we determine $m(k,n;s)$ for every $s\in \Z_n$.

\section{Proof of \cref{thm:main}}\label{sec:proof}

In this section, we prove \cref{thm:main}.  
We first recall Weil's bound for multiplicative character sums; see, for example \cite[Theorem 5.41]{LN97}.

\begin{lemma}[Weil's bound]\label{Weil}
Let $\chi$ be a multiplicative character of $\F_q$ of order $d>1$, and let $f \in \F_q[x]$ be a monic polynomial of positive degree that is not a $d$-th power of a polynomial. 
Let $n$ be the number of distinct roots of $f$ in its
splitting field over $\F_q$. Then for any $a \in \F_q$,
$$\bigg |\sum_{x\in\mathbb{F}_q}\chi\big(af(x)\big) \bigg|\le(n-1)\sqrt q .$$
\end{lemma}

The following lemma is a generalization of \cite[Theorem 10]{Gyarmati01} proved by Gyarmati, where she established a similar result over prime fields. We include a short proof for the sake of completeness.

\begin{lemma}\label{lem: Gyarmati} 
Let $f(x)\in\F_q[x]$ be a polynomial of degree $1<m<q$.
Assume that for every integer $d>1$ with $d\mid(q-1)$, the polynomial $f$ is not a constant multiple of a $d$-th power in $\F_q[x]$.
Let $A,B\subset \F_q^*$ and assume that
\[
|A||B| > q\Big(\frac{q-1}{q-m}\Big)^2 (m-1)^2.
\]
Then there exist $a\in A$, $b\in B$, and $x\in\F_q$ such that $ab=f(x)$.
\end{lemma}

\begin{proof}
Define
\[
N:=\#\{(a,b,x)\in A\times B\times \F_q:\ ab=f(x)\}.
\]
Let $\chi_0$ be the trivial character.
We extend every (multiplicative) character $\chi$ on $\F_q^*$ to $\F_q$ by setting $\chi(0)=0$.
Note that, for $y \in \F_q^*$, $\mathbf{1}_{\{y=1\}}=\frac{1}{q-1}\sum_{\chi \in \widehat{\F_q^*}}\chi(y)$, where $\widehat{\F_q^*}$ is the set of  characters on $\F_q^*$. 
Then, we have
\[N=\sum_{a \in A}\sum_{b \in B}\sum_{x \in \F_q}\mathbf{1}_{\{ab=f(x)\}}=\frac{1}{q-1}\sum_{a \in A}\sum_{b \in B}\sum_{x \in \F_q}\sum_{\chi \in \widehat{\F_q^*}}\chi(a^{-1}b^{-1}f(x)).\]
Let $r$ be the number of distinct $\F_q$-roots of $f(x)$.
Then
\begin{equation}\label{eq:N}   
(q-1)N-(q-r)|A||B|=\sum_{a \in A}\sum_{b \in B}\sum_{x \in \F_q}\sum_{\chi\ne \chi_0 }\chi(a^{-1}b^{-1}f(x)).
\end{equation}
By the orthogonality relation,  
\[\sum_{\chi \in \widehat{\F_q^*}}\Big|\sum_{a\in  A}\chi(a^{-1})\Big |^2=\sum_{\chi \in \widehat{\F_q^*}}\sum_{a,\widetilde{a}\in A}\chi(a^{-1})\overline{\chi(\widetilde{a}^{-1})}
=\sum_{a,\widetilde{a}\in A}\sum_{\chi \in \widehat{\F_q^*}}\chi(a^{-1}\widetilde{a})
=(q-1)|A|.
\]
It holds similarly for $B$.
This, together with the Cauchy--Schwarz inequality and Weil's bound, implies that
\begin{align*}
\left| \sum_{a \in A}\sum_{b \in B}\sum_{x \in \F_q}\sum_{\chi\ne \chi_0 }\chi(a^{-1}b^{-1}f(x))\right|
&\leq
\sum_{\chi\ne\chi_0} 
\left|\sum_{a\in A}\chi(a^{-1})\sum_{b\in B}\chi(b^{-1})\right|
\left|\sum_{x\in\F_q}\chi(f(x))\right| \\ 
&  \le \sqrt{\sum_{\chi \ne \chi_0}\Big|\sum_{a \in A}\chi (a^{-1})\Big|^2}\sqrt{\sum_{\chi \ne \chi_0}\Big|\sum_{b \in B}\chi(b^{-1})\Big|^2}(m-1)\sqrt{q}\\
& \le (m-1)(q-1)\sqrt{q|A||B|}.
\end{align*}
Thus equation~\eqref{eq:N} implies that
\[\left|(q-1)N-(q-r)|A||B|\right| \le (m-1)(q-1)\sqrt{q|A||B|}.\]
Since $q-r\ge q-m$, we have
\[(q-1)N \ge (q-m)|A||B| - (m-1)(q-1)\sqrt{q|A||B|}.\]
Since 
\[
|A||B|> q\Big(\frac{q-1}{q-m}\Big)^2 (m-1)^2,
\]
we have $N>0$, as required.
\end{proof}

Using \cref{lem: Gyarmati}, we have the following corollary.

\begin{corollary}\label{cor: h} 
Let $f(x)\in\F_q[x]$ be a polynomial of degree $1\le m<q$.
Assume that $f(x)\in \F_q[x]$ is not a constant multiple of a $d$-th power of a polynomial in $\F_q[x]$ for any divisor $d$ of $(q-1)$ with $d\geq 2$. Let $A_1, A_2,\ldots, A_k$ be subsets of $\F_q^*$ (not necessarily distinct).
If $|A_i| \ge 8C_m\sqrt q + 2k+2$ for each $1\le i\le k$, 
where 
\[
C_m:= (m-1)\frac{q-1}{q-m},
\]
then there exist distinct elements $a_1,a_2,\ldots, a_k$ such that $a_i \in A_i$ for $1\leq i \leq k$, and $$a_1\cdots a_k=f(x_0)$$ for some $x_0 \in \F_q$. 
\end{corollary}

\begin{proof}
If $m=1$, then $f$ is a bijection on $\F_q$, thus $f(\F_q)=\F_q$.
We choose $a_1\in A_1$ arbitrarily, and for $2\le j\le k$ choose
$a_j\in A_j\setminus\{a_1,\dots,a_{j-1}\}$, which is possible since at step $j$ we exclude at most $j-1\le k-1$ elements.
Set $y:=a_1\cdots a_k\in \F_q^*$.
Since $f(\F_q)=\F_q$, there exists $x_0\in \F_q$ such that $f(x_0)=y$, as required. 

Next, assume that $m\geq 2$. We consider two cases. 

\textbf{Case 1:} $k=2$. Choose disjoint subsets $X'\subset A_1$ and $Y'\subset A_2$
with \[
|X'| \ge \Big\lfloor\frac{|A_1|}{2}\Big\rfloor
\qquad \text{and} \qquad
|Y'| \ge \Big\lceil\frac{|A_2|}{2} \Big\rceil.
\]
Then $X'\cap Y'=\varnothing$, and
\[
|X'||Y'| \ge \Big\lfloor \frac{|A_1|}{2} \Big \rfloor \cdot \Big\lceil \frac{|A_2|}{2} \Big\rceil
\ge \frac{|A_1|-1}{2} \cdot \frac{|A_2|}{2}
\ge \frac{(8C_m\sqrt q+2k+1)(8C_m\sqrt q +2k+2)}{4}.
\]
In particular, since $k\ge2$, the right-hand side is at least $16C_m^2q$, and so
\[
|X'||Y'| > 16C_m^2 q
 \ge  q\Big(\frac{q-1}{q-m}\Big)^2(m-1)^2.
\]
By applying \cref{lem: Gyarmati} to $f$ and $(X',Y')$, we have
$a_1\in X'$, $a_2\in Y'$, and $x_0\in\F_q$ such that $a_1a_2=f(x_0)$. In particular, $a_1\ne a_2$ since $X'\cap Y'=\varnothing$.

\smallskip

\textbf{Case 2:} $k\ge 3$. We choose $a_1\in A_1$ arbitrarily, and for $2\le j\le k-2$ choose
$a_j\in A_j\setminus\{a_1,\dots,a_{j-1}\}$, which is possible since at step $j$ we exclude at most $j-1\le k-3$ elements. Let $c:=(a_1a_2\cdots a_{k-2})^{-1}\in\F_q^*$ and define
\[
X:=A_{k-1}\setminus\{a_1,\dots,a_{k-2}\},\qquad
Y:=A_k\setminus\{a_1,\dots,a_{k-2}\}.
\]
Then, 
\[
|X| \ge |A_{k-1}|-(k-2)\ge 8C_m \sqrt q + k+2, \qquad
|Y| \ge |A_k|-(k-2)\ge 8C_m \sqrt q + k+2.
\]
Choose disjoint subsets $X_0\subset X$ and $Y_0\subset Y$ such that
\[
|X_0| \ge \Big\lfloor \frac{|X|}{2} \Big\rfloor
\qquad \text{and} \qquad
|Y_0|\ge \Big\lceil \frac{|Y|}{2} \Big\rceil.
\]
In particular, $|X_0|\ge (|X|-1)/2 > 4C_m\sqrt q$ and $|Y_0|\ge |Y|/2 > 4C_m\sqrt q$ and thus
\[
|X_0||Y_0| > 16C_m^2 q \ge q \Big(\frac{q-1}{q-m}\Big)^2(m-1)^2.
\]
Moreover, since $c\in\F_q^*$, the polynomial $cf$ is not a constant multiple of a $d$-th power for any $d\mid(q-1)$ with $d>1$.
Thus, by applying \cref{lem: Gyarmati} to $cf$ and $(X_0,Y_0)$, there exist $a_{k-1}\in X_0$, $a_k\in Y_0$, and $x_0\in\F_q$ such that $a_{k-1}a_k = c f(x_0)$.
Therefore, we have $a_1\cdots a_k=f(x_0)$ with $a_1,\dots,a_k$ distinct, as required.
\end{proof}

Now we are ready to present the proof of \cref{thm:main}.

\begin{proof}[Proof of \cref{thm:main}]
We first prove an upper bound for $F_k(q;h)$.
Let $A \subset \F_q^*$ such that $a_1\cdots a_k \ne Cf(x)^{\ell}$ for any distinct $a_1,\ldots,a_k \in A$ and for any $x \in \F_q$.
Then, we decompose 
\[
A=\bigcup_{i=0}^{n-1} A_i,
\]
where $A_i=A \cap g^i H$ for $0\leq i \leq n-1$. 

Note that $H=\left< g^n \right>=\left< g^\ell \right>=(\F_q^*)^\ell$ and $|H|=(q-1)/n$.
Therefore, for each $a\in A_i\subset g^iH$, we may write $a=g^i x^\ell$ for some $x\in\F_q^*$, and choosing one such $x$ for each $a\in A_i$ gives a set $B_i \subset \F_q^*$ with $|B_i|=|A_i|$ and $A_i=\{g^ix^\ell :x\in B_i\}$. Let $M:=M(k,m)$ be a constant such that \cref{cor: h} applies whenever
$|A_i|\ge M\sqrt q$ for all $i$ (for example, one may take $M=8C_m+2k+2$).
Define $$B=\{0\le i \le n-1 \colon |A_i|\ge M\sqrt{q}\} \subset \Z_n.$$
Next, we prove the following key claim.
\begin{claim}
$s\notin kB$.    
\end{claim}
\begin{poc}
Suppose otherwise that $s\in kB$. 
Then, there exist indices $i_1,\dots,i_k\in B$ (not necessarily distinct) such that
\[
i_1+\cdots+i_k \equiv s \pmod n.
\]
Let $i_1+\cdots+i_k=tn+s$ for some $t \in \Z$.
Note that there exists an integer $t'$ such that 
\begin{equation}\label{eq:t'}
(g^{t'})^{\ell}=\frac{g^{tn+s}}{C}.
\end{equation}
Indeed, since $C\in g^sH$, we may write $C=g^sh_0$ for some $h_0\in H$.
Then,
\[
\frac{g^{tn+s}}{C}=\frac{g^{tn}}{h_0}\in H.
\]
Since the index of $H$ is $n$, every element of $H$
is an $\ell$-th power in $\F_q^*$.
Thus, there exists $u\in\F_q^*$ with $u^\ell=g^{tn+s}/C$, and write $u=g^{t'}$. This gives our required $t'$.

Since $\ell$ is chosen maximal in the representation $h=Cf^\ell$,
the polynomial $f$ is not a constant multiple of a $d$-th power in $\F_q[x]$
for any integer $d|(q-1)$ with $d>1$.
Also, since $|B_i|=|A_i| \ge M\sqrt{q}$ for each $i$, by \cref{cor: h} applied to the polynomial $f(x)/g^{t'}$ and sets $B_{i_1},B_{i_2},\ldots, B_{i_k}$, there exist distinct $x_1, x_2,\ldots, x_k$, where $x_j\in B_{i_j}$ for $1\leq j \leq k$, such that
\[x_1 \cdots x_k = \frac{f(x_0)}{g^{t'}},\]
for some $x_0 \in \F_q$.
Thus, by equation~\eqref{eq:t'}, 
\[g^{i_1}x_1^\ell  g^{i_2}x_2^\ell \cdots g^{i_k}x_k^\ell  = g^{tn+s} (x_1 \cdots x_k)^{\ell}=C(f(x_0))^\ell.\] 
Note that $g^{i_j}x_j^\ell \in A_{i_j}\subset A$ for each $1\leq j\leq k$. Moreover, these $k$ elements are distinct. Indeed, let $1\leq j_1,j_2\leq k$ with $j_1\neq j_2$. If $i_{j_1}\neq i_{j_2}$, then $g^{i_{j_1}}x_{j_1}^{\ell}$ and $g^{i_{j_2}}x_{j_2}^{\ell}$ are in different cosets of $H$ so they are different; if $i_{j_1}=i_{j_2}$, then since $x_{j_1}\neq x_{j_2}$, by the construction of $B_{i_{j_1}}$, we know $g^{i_{j_1}}x_{j_1}^{\ell}\neq g^{i_{j_2}}x_{j_2}^{\ell}$.
This contradicts the defining property of $A$.
Thus, $s\notin kB$.
\end{poc}

Thus, by the definition of $m(k,n;s)$, we obtain
$ |B|\le m(k,n;s)$. 
We conclude that \[|A| =\sum_{i\in B}|A_i|+\sum_{i\notin B}|A_i|\leq |B|\cdot \frac{q-1}{n}+O(\sqrt{q})\le \frac{m(k,n;s)}{n}\cdot q+O(\sqrt{q}),\] as required.

\medskip

Now, we obtain a lower bound of $F_k(q;h)$.
Let $B_0\subset \Z_n$ with $|B_0|=m(k,n;s)$ and $s\notin kB_0$.
Define
\[ A_0:=\bigcup_{i\in B_0} g^iH \subset \F_q^*. \]
Then
\[|A_0|=|B_0| |H|=m(k,n;s) \cdot\frac{q-1}{n}. \]
To finish our proof, we show that for any distinct $a_1,\dots,a_k\in A_0$ and any $x\in\F_q$, we have
$a_1\cdots a_k\ne h(x)$.
For each $1\le j \le k$, choose $i_j\in B_0$ such that $a_j\in g^{i_j}H$.
Then, $a_1\cdots a_k \in g^{i_1+\cdots+i_k}H$,
so there exists $t\in kB_0$ such that $a_1\cdots a_k\in g^tH$.
Note that $h(x) \in g^s H$ for all $x \in \F_q$.
Since $s\notin kB_0$, we have $t\ne s$ in $\Z_n$, and therefore $g^tH\ne g^sH$.
This implies $a_1\cdots a_k\ne h(x)$ for all $x\in\F_q$.
\end{proof}

\begin{remark}\label{rem:structure}
Our proof shows that if $A\subset \F_q^*$ satisfies ($\star$) for $h$ in \cref{defn} and satisfies
\[
|A|\ge \frac{m(k,n;s)}{n}\cdot q - O(\sqrt q),
\]
then the set $A$ consists of $m(k,n;s)$ cosets of $H$, up to $O(\sqrt q)$ exceptional elements.
More precisely, there exists a set $B_0\subset \Z_n$ with $|B_0|=m(k,n;s)$ and $s\notin kB_0$
such that
\[
\left|A\ \triangle \bigcup_{i\in B_0} g^iH\right|=O(\sqrt q).
\]
In other words, if $m(k,n;s)>0$, then extremal sets are essentially unions of $m(k,n;s)$ cosets of $H$. 
\end{remark}

\begin{remark}
    We finally remark that the natural $k$-set analogues of Lemma~\ref{lem: Gyarmati} and Corollary~\ref{cor: h} do not hold. More precisely, if one replaces the hypotheses by $\prod_{i\in [k]} |A_i| \ge \Omega(q m^k)$ and $|A_i| \geq \Omega(q^{1/k})$, respectively, then the resulting statements fail in general. Here, the implicit constants only depend on $m=\deg f$ and $k$. Assume for contradiction that these analogues hold. Then, by following the proof of Theorem~\ref{thm:main}, one can deduce that 
    $$
        F_k(q; h) = \frac{m(k, n;s)}{n}\cdot q + O(q^{1/k}).
    $$
    However, Remark~\ref{rmk:rmk} $(3)$ guarantees that there exists a polynomial $h$ and infinitely many prime powers $q$ for which $m(k,n;s)=0$, while $F_k(q; h) = \Theta(\sqrt{q})$ for each fixed $k \geq 2$. Hence, taking $k$ strictly larger than two yields a contradiction.

\end{remark}

\section{An Estimate for $m(k,n;s)$}\label{sec: sec4}

In this section, we estimate $m(k,n;s)$ for each $s\in \Z_n$, where $k\ge2$ and $n$ are positive integers.
We obtain an upper bound from Bajnok's work \cite{B15} and a lower bound by a proper explicit construction.

\begin{proposition}
Let $k\ge2$ and $n$ be positive integers. Then, for each $s\in \Z_n$, we have
$$
\max_{d\mid n}\Big(\Big\lfloor\frac{d-1-\gcd(d,k)}{k}\Big\rfloor+1\Big)\frac{n}{d} \le m(k,n;s)\le \max_{d\mid n}\Big(\Big\lfloor\frac{d-2}{k}\Big\rfloor+1\Big)\frac{n}{d}.
$$     
\end{proposition}

\begin{proof}
The upper bound is immediate from \cite[Theorem 6]{B15}.
Let $\chi(\Z_n,k)$ be the minimum value of $m$, if exists, for which the $k$-fold sumset of every $m$-subset of $G$ equals $G$ itself.
Then, every subset $B\subset\Z_n$ with $|B|\ge \chi(\Z_n,k)$ satisfies
$kB=\Z_n$, so $m(k,n;s)\le \chi(\Z_n,k)-1$.
The formula for $\chi(\Z_n,k)-1$ can be found in \cite[Theorem 6]{B15}, which coincides with the upper bound.

We now consider a lower bound of $m(k,n;s)$. 
Fix a divisor $d\mid n$ and let $K:=d\Z_n\le \Z_n$ be the subgroup of index $d$.
Let $\pi:\Z_n\to \Z_n/K \cong \Z_d$ be the canonical projection.
Set $r:=\gcd(d,k)$ and define
\[
t:=\Big\lfloor\frac{d-1-r}{k}\Big\rfloor+1.
\]
Choose a subset $T\subset \Z_d$ of size $t$ as
\[ T=\{a,a+1,\dots,a+t-1\}\subset \Z_d, \]
for some $a\in \Z_d$.
Let $B:=\pi^{-1}(T)\subset \Z_n$.
Then $|B|=|T| |K|=\frac{tn}{d}$.

We claim that $s\notin kB$ for a suitable choice of $a$.
Indeed, since $\pi$ is a homomorphism, we have $\pi(kB)=kT$.
In addition, since $T$ is an interval of length $t$, we have
\[
kT=\{ka,ka+1,\dots,ka+k(t-1)\},
\]
so $|kT|=k(t-1)+1\le d-r$ by the definition of $t$.
Let
\[
F:=\pi(s)-\{0,1,\ldots,k(t-1)\}\subset \Z_d.
\]
Then $\pi(s)\notin kT$ is equivalent to $ka\notin F$.

Now, the set of possible values of $ka$ as $a$ varies is exactly $k\Z_d=r\Z_d$.
Since $r\Z_d=\{0,r,2r,\ldots,d-r\}$, any interval that contains all of $r\Z_d$ must have length at least $d-r+1$.
Thus, $F$ cannot contain all elements of $r\Z_d$ since $|F|=k(t-1)+1\le d-r$.
So, we may choose $y\in r\Z_d\setminus F$.
Since $y\in r\Z_d=k\Z_d$, there exists $a\in\Z_d$ such that $ka=y$.
For this choice of $a$, we have $ka\notin F$, hence $\pi(s)\notin kT$,
and therefore $s\notin kB$.
\end{proof}

When $\gcd(k,n)=1$, we obtain the following corollary immediately.

\begin{corollary}\label{cor:mkns}
Let $k\ge2$ and $n$ be positive integers.
If $\gcd(k,n)=1$, then for each $s\in \Z_n$, we have
\begin{equation}\label{eq: gcd(k,n)}
m(k,n;s)=\max_{d \mid n} \left(\bigg\lfloor\frac{d-2}{k}\bigg\rfloor+1\right) \frac{n}{d}.
\end{equation}   
\end{corollary}

When $s=0$, this quantity $m(k,n;s)$ asks for the maximum size of a \emph{$k$-zero-free subset} of $\Z_n$. It was first studied by Sargsyan \cite{SS13}. He was not able to determine the exact values of $m(k,n;0)$ in general, although several lower and upper bounds are given in \cite[Section 3]{SS13}. 
Instead, he determined an exact formula for $m(k,n;0)$ when $\gcd(k,n)=1$.
In this case, his result also can yield the same formula for $m(k,n;s)$, described in \cref{cor:mkns}.

Indeed, for $t\in\Z_n$ we have $k(B+t)=kB+kt$.
Since $\gcd(k,n)=1$, multiplication by $k$ is a bijection on $\Z_n$, so for any $s$ we can choose
$t\equiv -k^{-1}s\pmod n$.
Note that 
$0 \notin k(B+t)$ is equivalent to $-kt=s \notin kB$.
Since $|B|=|B+t|$, $m(k,n;s)=m(k,n;0)$ for all $s \in \Z_n$.
When $s=0$, \cite[Corollary in Section 3]{SS13} gives an exact formula of $m(k,n;0)$, which coincides with equation \eqref{eq: gcd(k,n)}.

\section*{Acknowledgments}
The authors thank Péter Pál Pach for helpful discussions. The authors thank the anonymous referees for their valuable comments and suggestions. C.H. Yip thanks Institute for Basic Science for hospitality on his visit, where part of this project was discussed.
H. Lee was supported by the National Research Foundation of Korea (NRF) grant funded by the Korea government(MSIT) No. RS-2023-00210430, and the Institute for Basic Science (IBS-R029-C4).
C.H. Yip was supported in part by an NSERC fellowship.
S. Yoo was supported by the Institute for Basic Science (IBS-R029-C1). 

\bibliographystyle{abbrv}
\bibliography{references}

\end{document}